\documentclass[12pt,A4]{article}

\DeclareSymbolFont{AMSb}{U}{msb}{m}{n}
\DeclareSymbolFontAlphabet{\Bbb}{AMSb}

\usepackage[ansinew]{inputenc} 
\usepackage{latexsym}

\usepackage{amstext,amsfonts,amsmath,amsthm,graphicx,amssymb,amscd,epsfig}
\usepackage{rotating}

\newtheorem{theorem}{Theorem}[section]

\newtheorem{definition}[theorem]{Definition}
\newtheorem{lemma}[theorem]{Lemma}
\newtheorem{corollary}[theorem]{Corollary}
\newtheorem{proposition}[theorem]{Proposition}

\newtheorem{example}[theorem]{Example}

\newcommand{\norm}[1]{\left\Vert#1\right\Vert}
\newcommand{\abs}[1]{\left\vert#1\right\vert}

\newcommand{\R}{\mathbb{R}}
\newcommand{\Z}{\mathbb{Z}}

\newcommand{\spec}{\operatorname{{Spec}}}

\newcommand{\Emb}{\textrm{Emb}}
\newcommand{\Diff}{\textrm{Diff}}
\newcommand{\Hom}{\textrm{Hom}}
\newcommand{\Fix}{\textrm{Fix}}
\newcommand{\Per}{\textrm{Per}}
\newcommand{\ind}{\textrm{ind}}
\newcommand{\dg}{\textrm{deg}}

\title{\bf Global Saddles for Planar Maps}
\author{B. Alarc\'on$^a$, S.B.S.D. Castro$^{b,c}$ and I.S. Labouriau$^{b,d}$}
\date{{}}

\begin{document}
\maketitle

\noindent{\small $a$ Departamento de Matem\'atica Aplicada, Instituto de Matem\'atica e Estat\'\i stica, Universidade Federal Fluminense, \\
Rua Professor Marcos Waldemar de Freitas Reis, S/N, Bloco H, \\
Campus do Gragoat\'a, 
CEP 24.210 -- 201, S\~ao Domingos --- Niter\'oi, RJ, Brasil}\\
{\small $b$ Centro de Matem\'atica da Universidade do Porto,\\ Rua do Campo Alegre 687, 4169-007 Porto, Portugal.}\\
{\small $c$ Faculdade de Economia do Porto, \\ Rua Dr. Roberto Frias, 4200-464 Porto, Portugal.}\\
{\small $d$ Faculdade de Ci\^encias da Universidade do Porto,\\ Rua do Campo Alegre 687, 4169-007 Porto, Portugal.}
\bigbreak

\begin{abstract}
We study the dynamics of planar diffeomorphisms having a unique fixed point that is a hyperbolic local saddle. We obtain sufficient conditions under which the fixed point is a global saddle.
We also address the special case of $D_2$-symmetric maps, for which we obtain a similar result for $C^1$ homeomorphisms. Some applications to differential equations are also given.
\end{abstract}
\vfill
\thispagestyle{empty}
\newpage

\section{Introduction}
The study of global dynamics has long been of interest.
Particular attention has been given to the question of inferring global results from local behaviour, when a unique fixed point is either a local attractor or repellor.
One famous instance is the Markus-Yamabe conjecture \cite{vanDenEssen}, its proof for dimension 2
by Gutierrez \cite{Gutierrez}, and several counterexamples for continuous time higher dimensional dynamics, and  for discrete time in dimension greater than or equal to 2.
Interest in this has then extended to planar discrete dynamics. The case when the unique fixed point is a local saddle was addressed in \cite{sillapepe} for $C^1$ vector fields. However, the problem of characterising the existence of a global saddle for planar diffeomorphisms is still open.

The presence of symmetry in a dynamical system creates special features that may be used to obtain global results.
Planar dynamics with symmetry, when the fixed point is either an attractor or a repellor, has been addressed by the authors in \cite{AlarconDenjoy,SofisbeSzlenk,SofisbeGlobal}.
There results, as well as those  without extra assumptions on symmetry, ignore the important case when the fixed point is a local saddle.

A local saddle for a map $f:\R^2\rightarrow\R^2$ is a fixed point of $f$, without loss of generality the origin, such that $Df(0)$ has eigenvalues $\lambda$, $\mu$ satisfying $0<|\lambda|<1<|\mu|$.
When $\lambda, \mu>0$, the local saddle is called direct.
Otherwise, it is called twisted.

When we extend a local saddle globally, we ask for the local stable and unstable manifold to extend to infinity without homoclinic contacts.
We stress that our concept of a global saddle is weaker than demanding that it be globally conjugated to a linear saddle.

In the present article we conclude the study of global discrete planar dynamics by looking at maps with a saddle at the origin, with results concerning maps both with and without symmetry.

The non-symmetric case has been addressed by Hirsch \cite{sillaHirsch} for direct saddles of orientation preserving diffeomorphisms such that every fixed point has a negative index. We obtain results  for diffeomorphisms without period-2 points and include the case of twisted saddles.
A simple example shows that our hypotheses are minimal.

Additionally, in the symmetric case, our results include fixed points of positive index.
We also relax the diffeomorphism hypothesis to include homeomorphisms of class $C^1$.
When the symmetry group possesses two reflections, we show that the stable and unstable manifolds divide the plane in four connected components, that are either $f$-invariant or permuted by the dynamics.

The article is organised as follows: the next section contains definitions and preliminary results, Section~\ref{secGlobalSaddle} addresses dynamics without symmetry and ends with  an example showing that our hypotheses are minimal.
Section \ref{secSymm}, after a few results on symmetric maps, deals with the symmetric case. Section \ref{secApp} gives some applications of the main results to differential equations.

\section{Background and definitions}

In this article we  work with sets of planar maps, for which we introduce some notation:
$\Emb(\R^2)$ for continuous and injective maps, 
$\Hom(\R^2)$ for homeomorphisms and
$\Diff(\R^2)$ for $C^1$ diffeomorphisms.
A superscript $+$ as in  $\Diff^+(\R^2)$ indicates the subset of orientation preserving maps.

We are concerned with local and global saddles.
We start with the definitions of  local saddle, stable and unstable local manifolds and homoclinic contacts for $C^ 1$ maps, as in Hirsch  \cite{sillaHirsch}, and then proceed to  define (topological) global saddle.

Given a $C^1$ map  $f: \R^2 \to \R^2$ with  $0\in \Fix(f)$, we will say that $0$ is a
{\em local saddle} if the derivative $Df(0)$ has eigenvalues $\lambda, \mu \in \R$ satisfying $0<|\lambda|<1<|\mu|$. If both eigenvalues of $Df(0)$ are strictly positive, $\lambda, \mu >0$, we will say that $0$ is a
 local {\em direct saddle}. In other cases $0$ will be called a
 local  {\em twisted saddle}.
 Note that if $0$  is a local saddle for $f$, then it is a local direct saddle for $f^2$.

The Grobman-Hartman theorem implies that if $0$ is a local saddle then there is an open neighbourhood $U$ of $0$ and a
homeomorphism $h$ defined in $U$, with $h(0)=0$,  that conjugates $f$ into  a linear map with eigenvalues $\lambda, \mu$.
Thus $h^{-1}$ maps the eigenspaces of the linear map into two curves in $U$, the {\em local stable} and {\em local unstable manifolds} of $f$.

The (global) {\em stable curve}, containing the local stable manifold, is defined as
$$
W^{s}
=W^{s}(0,f)=\{x \in \R^2 : \lim_{n\to \infty} f^{n}(x)=0\}\ .
$$
If $f$ is invertible, then the (global) {\em unstable curve} is given by $W^{u}=W^{u}(0,f)=W^{s}(0,f^{-1})$.
Each one of these curves may be parametrised by a $C^1$ immersion $\tau: \R \to \R^2$ with
$\tau(0)=0$. The images of $(-\infty,0]$ and $[0,+\infty)$ are the two {\em stable  branches} ({\em unstable  branches}, respectively) at $0$. These branches will be denoted by $\beta^{-}$ and $\beta^{+}$, respectively.

For the branch $\beta^{+}$ and $\beta^{-}$ parametrised by a continuous bijection  $\zeta: [0,+\infty) \to \beta^{+}$ and $\zeta': (-\infty,0] \to \beta^{-}$, respectively,  consider the limit sets
$$
\mathcal{L}(\beta^{+})=\bigcap_{t\geq 0} \overline{\zeta([t, +\infty))}, \quad \mathcal{L}(\beta^{-})=\bigcap_{t\leq 0} \overline{\zeta'((-\infty,t])}
$$
that do not depend on the choice of parametrisation $\zeta$ and $\zeta'$, respectively.
This set  is closed and $f$-invariant.
We define the {\em  limit set} $\mathcal{L}(W^u)$ as the union of the limit sets of the two branches of $W^u$, the definition of  $\mathcal{L}(W^s)$ is analogous.

Given a continuous map $f: \R^2 \to \R^2$, we say that $p$ is a \emph{non-wandering point} of $f$ if for every neighbourhood $U$ of $p$ there exists an integer $n>0$ and a point $q\in U$ such that $f^n(q)\in U$. We denote the set of non-wandering points by $\Omega(f)$. We have
$$
\Fix(f)\subset \Per(f) \subset \Omega(f),
$$
where $\Fix(f)$ is the set of fixed points of $f$,  and $\Per(f)$ is the set of periodic points of $f$.

Let $\omega(p)$ be the set of points $q$ for which there is a sequence $n_j\to+\infty$ such that 
{  $f^{n_j}(p)\to q$.}
If $f\in Hom(\R^2)$ then $\alpha(p)$ denotes the set $\omega(p)$ under $f^{-1}$.


A point of the closed invariant set
$$
(W^u \cap W^s \setminus \{0\}) \cup (\mathcal{L}(W^u) \cap \overline{W^s}) \cup (\mathcal{L}(W^s) \cap \overline{W^u}).
$$
is called a  {\em homoclinic contact}.

\begin{definition}\label{def:global_saddle} Let $f: \R^2 \to \R^2$ be a $C^1$ homeomorphism such that $f(0)=0$.
We say that $0$ is a {\em global (topological) saddle} if $0$ is a local saddle, there are no homoclinic contacts  and $W^{s}(0,f)$, $W^{u}(0,f)$ are unbounded sets such that
 for all $p\notin  W^{s}(0,f)\cup W^{u}(0,f)\cup \{0\}$ both $\norm{f^n(p)}\to \infty$ and $\norm{f^{-n}(p)}\to \infty$ as $n$ goes to $\infty$.
\end{definition}

If the local saddle in Definition~\ref{def:global_saddle} is direct, then we talk of a {\em global direct saddle}.

Note that a global (topological) saddle may not be conjugated to the linear saddle because more complex features can appear.
{  For instance, a  global topological saddle may have more than the four elliptic components at infinity that occur for 
a linear saddle.}
See Figure \ref{figuraReebComp}. {  An interesting discussion, as well as illustrations, for the continuous-time case may be found in \cite{sillapepe}.}

\begin{figure}[hh]
\centering

\includegraphics[width=5cm]{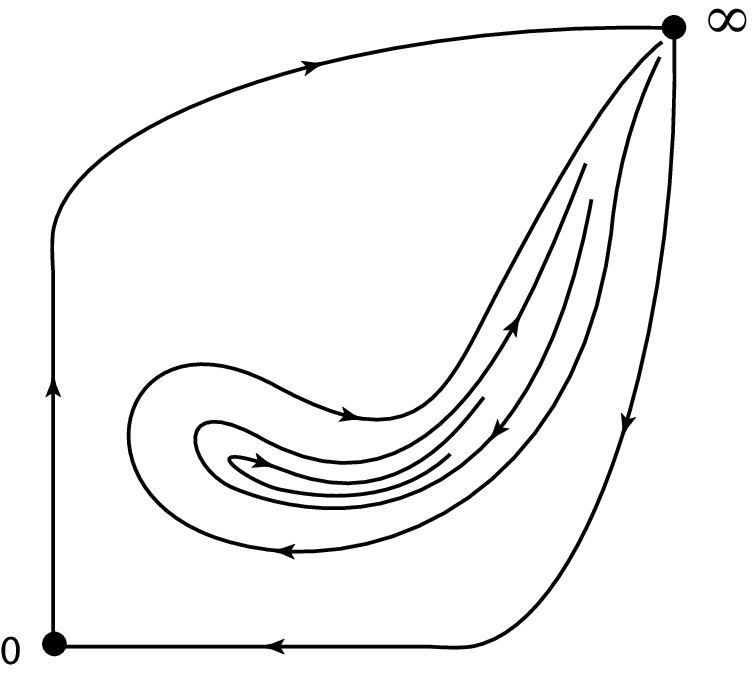} \qquad\qquad
\includegraphics[width=5cm]{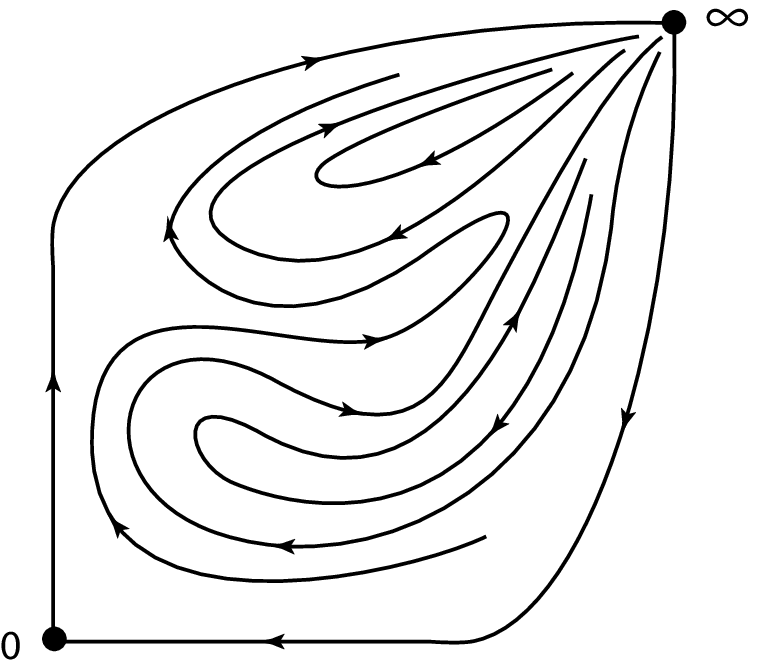}
\caption{{Dynamics on the first quadrant of two non conjugated global saddles.
 The linear saddle  (left)  has four elliptic components at infinity.
A global saddle (right) not conjugated to the linear saddle may have more elliptic components, for instance, it may have an extra component on the first quadrant.
}
\label{figuraReebComp}}
\end{figure}


\begin{definition} A map $f\in \Hom(\R^2)$ is {\em free} if, given any topological disk $D\subset \R^2$ such that $f(D)\cap D=\emptyset$, then $f^n(D)\cap f^m(D)=\emptyset$ for each $n,m\in \Z$, $n\neq m$.
\end{definition}
%
%

Let $f\in \Hom(\R^2)$  and $\tilde{f}\in \mathbb{S}^2$ be the extension of $f$ to $\R^2\cup\{\infty\}\cong \mathbb{S}^2$ by letting $\tilde{f}(\infty)=\infty$. The next result is a version of  Lemma 3.4 in \cite[page 456]{Brown}.

\begin{lemma} \label{freetridynHom}  If $\tilde{f}\in \Hom(\mathbb{S}^2)$ is free, then

\begin{itemize}
\item[(a)] $\Fix(\tilde{f})\neq \emptyset $.
\item[(b)] For each $x\in \mathbb{S}^2$, $\limsup_{|n|\to\infty} \tilde{f}^n(x)\subset \Fix(\tilde{f})$.
\item[(c)] If $\Fix(\tilde{f})$ is totally disconnected then for each $x\in \mathbb{S}^2$ there exist points $\alpha(x), \omega(x)$ (not necessarily distinct) of $\Fix(\tilde{f})$ such that $\lim_{n\to\infty} \tilde{f}^n(x)=\omega(x)$ and $\lim_{n\to -\infty} \tilde{f}^n(x)=\alpha(x)$.
\end{itemize}
\end{lemma}

For $f\in \Hom(\R^2)$ we write $\omega(p)=\infty$ when $\omega(p)=\infty$ for $\tilde{f}\in \mathbb{S}^2$.
Analogously, we introduce $\alpha(p)=\infty$.
Hence,  $\omega(p)=\infty$ means that $\norm{f^n(p)}\to \infty$ as $n$ goes to $ \infty$ and $\alpha(p)=\infty$ means that $\norm{f^n(p)}\to \infty$ as $n$ goes to $ -\infty$.



We say that $f\in \Hom(\R^2)$ has {\em trivial dynamics} if $\alpha(p), \omega(p)\subset \Fix(\tilde{f})$ for every $p\in \R^2$. Observe that in the case of a unique fixed point $q$ of $f$, $\Fix(\tilde{f})=\{q,\infty\}$. \\


The next result is a consequence of Lemma \ref{freetridynHom}.

\begin{proposition}
\label{freetridyn} Assume that $f\in \Hom(\R^2)$ is free, then it has trivial dynamics.
\end{proposition}

In particular, all non-wandering points of planar free homeomorphisms are fixed points.
So $\Fix(f)=\Omega(f)$ for all free $f\in \Hom(\R^2)$ --- more details and examples can be found in \cite{Brown}).
The  trivial dynamics property is analogous to the Poincar\'e-Bendixon Theorem for continuous time systems. When free becomes  too strong a condition, we shall use
the fixed point index.

\begin{definition} We define the {\em index} of a fixed point $p$ of a continuous map $f:\R^2\to \R^2$ as $$\ind (f,p)=\dg  (I-f,D),$$ where $I$ is the identity map in $\R^2$, and $D\in \R^2$ is a topological disc which is a neighbourhood of $p$, and $\Fix(f)\cap \partial{D}=\emptyset$ and $\dg  (I-f,D)$ is the Brouwer degree of the map $I-f$.

Let $L$ be an open simply connected subset of the plane such that $\Fix(f)\cap \partial L=\emptyset$, we define
$\ind(f,L)$ as $\sum_{p\in Fix(f)\cap L} \ind(f,p)$.
\end{definition}

The next theorem shows the relation between free homeomorphisms and degree theory.

\begin{theorem}[Brown \cite{Brown}, Theorem 5.7]\label{teoFreeHom} Assume that $f\in \Hom^{+}(\R^2)$ and that  for every Jordan curve $\Psi\subset \R^2\setminus \Fix(f)$ with $\widehat{\Psi}$ as bounded component we have
$$
\ind (f,\widehat{\Psi})\neq 1\ .
$$
Then $f$ is free.
\end{theorem}

%
%

If $f$ is differentiable denote its spectrum by $\spec(f)$. We have:

\begin{lemma}[Corollary 2 in \cite{Alarcon-orbitas-periodicas}]\label{corPtoFijo}
Let $f:\R^2\rightarrow\R^2$ be a differentiable map such that for some $\varepsilon>0$,
$\spec(f)\cap [1,1+\varepsilon[=\emptyset$, then $f$ has at most one fixed point.
\end{lemma}
%

Let $p\in \R^n$ and $f:\R^n \to \R^n$ be a continuous map. We denote by $\omega_2(p)=\{q\in \R^n : \lim \;f^{2n_k}(p)=q, \; \text{ for some sequence} \; 2n_k\to \infty \}$ the $\omega-$limit of $p$ with respect to $f^2$.
If $f$ is invertible then we define $\alpha_2(p)$ in a similar way.

\begin{lemma} [Lemma 3.1 in \cite{SofisbeGlobal}]\label{lemOmega2} Let $f:\R^n \to \R^n$  be a homeomorphism such that $f(0)=0$. Then, for $p\in \R^n$, the following hold:
\begin{enumerate}
\item[a)]  if $\omega_2(p)=\{0\}$, then   $\omega(p)=\{0\}$;
\item[ b)] if $\omega_2(p)=\infty$,   then
$\omega(p)=\infty$.
\end{enumerate}
\end{lemma}
%
%
%

The final ingredient to establish our results is the following:

\begin{theorem}[Hirsch, \cite{sillaHirsch}] \label{teoHirsch} Let $f\in \Diff^{+} (\R^2)$ be such that every fixed point is isolated and has index $\leq 0$. Then the following statements hold:

\begin{itemize}
\item[i)]
For every $x$, as $n$ goes to $\pm \infty$, either
$f^n(x)$ goes to a fixed point or $\norm{f^n(x)}\to \infty$.
\item[ii)]
For each direct saddle $p$, every homoclinic contact is
a fixed point different from $p$ and each branch at $p$ is
homeomorphic to $[0,\infty)$.
\item[iii)]
If the only fixed point
is a direct saddle $p$, then there are no homoclinic contacts and
every branch of $W^s(p)$ and of $W^u(p)$ is unbounded.
\end{itemize}
\end{theorem}

\section{Topological global saddle}\label{secGlobalSaddle}

We start with an immediate application of Theorem~\ref{teoHirsch}.

\begin{corollary}\label{coroHirsch}
Let $f\in \Diff (\R^2)$ be such that $\Fix(f)=\{0\}$ and $0$ is a local direct saddle.
Then $0$ is a global saddle.
\end{corollary}

\begin{proof}
Since $0$ is a local direct saddle, then $f$ preserves orientation and $0$ has negative index.
By i) and ii) in Theorem~\ref{teoHirsch}, a point $p$ not in $W^s(0,f)\cup W^u(0,f)\cup \{0\}$ is such that $\omega(p)=\alpha(p)=\infty$ and $W^s(0,f)$ and $W^u(0,f)$ have no homoclinic contact and are unbounded.
\end{proof}

\begin{proposition} Let $f\in \Hom(\R^2)$ be such that $f^2$ has trivial dynamics. If $\Fix(f)$ is a discrete set and $\Fix(f^2)=\Fix(f)$, then $f$ has trivial dynamics.
\end{proposition}

\begin{proof} Given $p\in \R^2$, since $f^2$ has trivial dynamics, then both $\omega_2(p)$ and
$\alpha_2(p)\subset  \Fix(f^2)\cup \{\infty\}=\Fix(f)\cup \{\infty\}$.
Suppose there exists a fixed point $q$ of $f$,  such that $\omega_2(p)=q$.
Then, as in Lemma \ref{lemOmega2}, $\omega(p)=q$. Suppose now that $\omega_2(p)=\infty$, then as in Lemma \ref{lemOmega2},  $\omega(p)=\infty$. Considering $f^{-1}$ we can prove in the same way that $\alpha(p)\subset \Fix(f)\cup \{\infty\}$.
\end{proof}

\begin{lemma} \label{propSaddleFree}
Let $f\in \Hom^{+}(\R^2)$ be of class $C^1$. If the origin is the unique fixed
point of $f$ and it is a local  direct saddle, then $f$ is free.
\end{lemma}

\begin{proof}
Let $\Psi$ be a Jordan curve such that $\Psi\cap \Fix(f)=\emptyset$. Thus, $\ind (f,\hat{\Psi})=-1$ if $0\in \hat{\Psi}$, where $\hat{\Psi}$ is the bounded connected component of $\R^2\setminus \Psi$. Moreover, if $0\notin \hat{\Psi}$, then
$\ind (f,\hat{\Psi})=0$  because $0$ is the unique fixed point. So
$f$ is free by Theorem \ref{teoFreeHom}.
\end{proof}

Recall that the uniqueness of a local direct saddle may be obtained from Lemma~\ref{corPtoFijo}.
Using this lemma and Lemma~\ref{propSaddleFree} we obtain the following:

\begin{proposition} \label{corolSaddleFree}
Let $f\in \Hom^{+}(\R^2)$ be of class $C^1$. Suppose that $0$ is a fixed point of $f$ which is a local  direct saddle.
If $\spec(f)\cap [1,1+\varepsilon[=\emptyset$, then $f$ is free.
\end{proposition}

%
%

\begin{corollary}\label{cor2Hirsch}
 Let $f\in \Diff(\R^2)$ such that $f(0)=0$ and for some $\varepsilon>0$, $\spec(f)\cap [1,1+\varepsilon[=\emptyset$. If the origin is a local direct saddle, then it is a global saddle.
%
\end{corollary}

\begin{proof} Follows by Lemma \ref{corPtoFijo} and Corollary~\ref{coroHirsch}.
\end{proof}

\begin{theorem} \label{teoSaddle} Let
$f\in \Diff(\R^2)$ be such that $\Fix(f^2)=\Fix(f)=\{0\}$ and 0 is a local saddle.
Then  $0$ is a global saddle.
\end{theorem}

\begin{proof}
Since $0$ is a local direct saddle for $f^2$ and $0$ is its unique fixed point, then 
$0$ is a global direct  saddle for $f^2$
by Corollary~\ref{coroHirsch}.
Moreover, by Lemma
\ref{lemOmega2}, $\omega(p)=\omega_2(p)$ and
$\alpha(p)=\alpha_2(p)$ for all $p\in \R^2$. So
$W^{s}(0,f)=W^{s}(0,f^2)$, $W^{u}(0,f)=W^{u}(0,f^2)$.
Then $0$ is a global  saddle for $f$.
\end{proof}

{  An example of a map satisfying the hypotheses of Theorem~\ref{teoSaddle} is given by
$f(x,y)=\left(\varphi(x),-2y\right)$, with $\varphi(x)=x\left(1+\arctan^2 x\right)/(4+\pi^2)$.
Note that  $\varphi(x)$ is odd, monotonic, $\lim_{x\to+\infty}\varphi(x)=+\infty$ and $0<\varphi(x)<x$ for $x>0$, implying that $0<\varphi^2(x)<\varphi(x)<x$. Hence, $\Fix(f^2)=\Fix(f)=\{0\}$.
}

The next example shows that Theorem \ref{teoSaddle} is false without the
hypothesis $\Fix(f^2)=\{0\}$ even when the map is orientation
preserving. That phenomenon appears when the saddle is not direct
because in this case the map interchanges the quadrants.

\begin{example} \label{exemploTwistedSaddle}
Consider the polynomial $p(x)=-ax^3+(a-1)x$ with $0<a<1$. Then the map $f: \R^2 \to \R^2$ given by $f(x,y)=(p(x),-2y)$ (with dynamics as in Figure \ref{figurasaddle}) verifies:
\begin{enumerate} \item $f\in \Diff^{+}(\R^2)$.   \item $\spec(f)\cap \R^{+}=\emptyset$.
\item $0$ is a twisted saddle.
 \item $\Fix(f)= \{0\}$.
 \item $\Fix(f^2)=\{-1,0,1\}\neq \{0\}$.
\item $\ind(f,\Psi)=+1$ or $0$.
\item $f$ is not free.
\end{enumerate}

\begin{figure}[hh]
\centering
\includegraphics[scale=0.5]{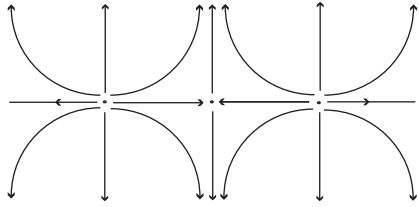}
\caption{A twisted saddle that is not free. The origin is a fixed point of $f$ and the points on the left and right are fixed points of $f^2$, symmetrically located around the origin.
\label{figurasaddle}}
\end{figure}
\end{example}

\section{Symmetric global saddle}\label{secSymm}

Let $\Gamma$ be a compact Lie group acting on $\R^2$, that is, a group which has the structure of a compact $C^{\infty}-$differentiable manifold such that the map $\Gamma \times \Gamma \to \Gamma$, $(x,y)\mapsto xy^{-1}$ is of class $C^{\infty}$. The following is taken from Golubitsky {\em et al.} \cite{golu2},  especially Chapter~XII,  to which we refer the reader interested in further detail.

Given a map $f:\R^2\longrightarrow\R^2$,
we say that $\gamma \in \Gamma$ is a \emph{symmetry} of $f$ if $f(\gamma x)=\gamma f(x)$.
We define the \emph{symmetry group} of $f$ as the biggest closed subset of $GL(2)$ containing all the symmetries of $f$. It will be denoted by $\Gamma_f$.

We say that $f:\R^2\to \R^2$ is  \emph{$\Gamma-$equivariant} or that $f$ {\em commutes} with $\Gamma$ if
$$
f(\gamma x)=\gamma f(x) \quad \text{ for all }\quad \gamma \in \Gamma.
$$
It follows that
every map $f:\R^2\to \R^2$ is equivariant under the action of its symmetry group, that is, $f$ is $\Gamma_f$-equivariant.

Let  $\Sigma$ be a subgroup of $\Gamma$. The {\em fixed-point subspace} of $\Sigma$ is
$$
\Fix (\Sigma) =\{p\in \R^2: \sigma p=p \; \text{ for all } \; \sigma \in \Sigma\}.
$$
If $\Sigma$ is generated by a single element $\sigma \in \Gamma$, we write $\Fix\langle\sigma\rangle$ instead of $\Fix (\Sigma)$.
We note that, for each subgroup $\Sigma$ of $\Gamma$, $\Fix (\Sigma)$ is invariant by the dynamics of a $\Gamma$-equivariant map (\cite{golu2}, XIII, Lemma 2.1).

For a group $\Gamma$ acting on $\R^2$ a non-trivial fixed point subspace arises when $\Gamma$ contains a reflection. By a linear change of coordinates we may take the reflection to be the {\em flip}
    $$
    \kappa \cdot (x,y) = (x, -y) .
    $$

%
%
%

In Alarc\'on {\em et al} \cite{SofisbeGlobal}, we provide a list of symmetry groups for which the corresponding equivariant maps may possess a local saddle.
There it is shown that the only symmetry groups
that admit a local saddle are $\Z_2(\langle-Id\rangle)$, $\Z_2(\langle\kappa\rangle)$ and $D_2=\Z_2^a\oplus\Z_2^b$. The  superscripts $a$ and $b$ indicate that the groups $\Z_2^a$ and $\Z_2^b$ are generated by two reflections, $a$ and $b$, on orthogonal lines.
Note that both $\Z_2(\langle\kappa\rangle)$ and $D_2=\Z_2^a\oplus\Z_2^b$ contain a reflection, 
and hence an $f$-invariant line.
A description of the admissible $\omega$-limit set of a point in some cases, when the symmetry group contains a flip is also given in \cite{SofisbeGlobal}.

The presence of at least one reflection in the symmetry group allows us to relax the hypotheses used in Section~\ref{secGlobalSaddle}
 in order to obtain a global saddle for $C^1$ homeomorphisms, not necessarily diffeomorphisms, and to extend the result to include twisted saddles. That happens because the hypotheses in
 Theorem~\ref{teoHirsch} are not necessary conditions.

\begin{theorem} \label{teorayoS}
Let $f\in \Hom(\R^ 2)$ be $C^1$ with symmetry group $\Gamma$ such that $\kappa \in \Gamma$,  the origin is a local saddle, and $\Fix(f)=\{0\}$. Suppose one of the following holds:
\begin{itemize}
\item[$a)$] $f$ is orientation preserving and $0$ is a local direct saddle;
\item[$b)$]  $\Fix(f^ 2)=\{0\}$.
\end{itemize}
Then the origin is a global saddle.

\end{theorem}

\begin{proof}
From the symmetry it follows that one of the global curves $W^s(0,f)$ or $W^u(0,f)$ is contained in
$\Fix(\kappa)$.
Without loss of generality, let  $W^s(0,f)\subset\Fix(\kappa)$.

Case a): 
It follows from Lemma~\ref{propSaddleFree} that $f$ is free and therefore has trivial dynamics.
Hence, $\omega(p)=\{0\}$ if and only if $p\in W^s(0,f)$, otherwise $\omega(p)=\infty$. Analogously, $\alpha(p)=\{0\}$ if and only if $p\in W^u(0,f)$, otherwise $\alpha(p)=\infty$.
This holds, in particular, for $p\ne 0$, in $W^u(0,f)$ and therefore $W^u(0,f)$ is unbounded, since $\omega(p)=\infty$.

An adaptation of Proposition~3.4 in \cite{SofisbeGlobal} shows the absence of homoclinic contacts, as follows:
let $r\ne 0, \infty$ be a homoclinic contact, that is 
$$
 r\in (W^u \cap W^s \setminus \{0\}) \cup (\mathcal{L}(W^u) \cap \overline{W^s}) \cup (\mathcal{L}(W^s) \cap \overline{W^u}).
 $$
Because $W^s=\Fix \langle\kappa\rangle$, we have:
$$
 \overline{W^s}=\mathcal{L}(W^s) =W^s\cup \{\infty\}.
$$

Because $f\in\Hom(\R^2)$ and $W^s=\Fix \langle\kappa\rangle$, we have $W^u\cap W^s=\{0\}$.

Assume $r\in \mathcal{L}(W^u) \cap {W^s}$. Then 
$$
\exists q_j\in W^u\qquad
\exists n_j\to +\infty\qquad
s.t. \quad f^{n_j}(q_j)\to r.
$$

Let $K=B_\varepsilon(r)$, $\varepsilon>0$ such that $0\not\in K$.
Then $\Fix \langle\kappa\rangle \cap K$ is an embedded
segment and $K\setminus\Fix \langle\kappa\rangle$ is the union of two disjoint disks
$W_1$ and $W_2$ homeomorphic to $\R^2$. 
Suppose without loss of generality that $f^{n_j}(q_j)\in W_1$.
Since $r$ is not a fixed point, taking $\varepsilon$ sufficiently small, there exists an open disk $V\subset W_1$  and a positive integer $n$, with $n\ge 2$,
such that for some $s\in V,$ we have that $f^n(s)\in V$, while
$V\cap f^{\ell}(V)=\emptyset $, for $\ell=1,2,\dots, n-1$. Then, by Theorem 3.3 in \cite{Murthy}, $f$ has a fixed point in $V$ which contradicts the uniqueness of the fixed point. So  $r\ne\infty$ is not a homoclinic contact.
\bigbreak

Case b):
{  Note that $r$ is a homoclinic contact for $f$ if and only if  $r$ is a homoclinic contact for $f^2$, because the invariant manifolds $W^s$ and $W^u$ for $f^2$ coincide with those for $f$.The result follows  from this}
and because $f^2$ satisfies the conditions of case a).
Recall that, by Lemma~\ref{lemOmega2}, and since $f^2$ has trivial dynamics, the $\omega$-limits of $f$ and $f^2$ coincide.
%
\end{proof}

{  Any one of the cases of Theorem~\ref{teorayoS} can be applied to show that 
$$
f(x,y)=\left(x\left(1+\arctan^2 x\right)/(4+\pi^2),2y \right)
$$ 
has a global saddle at the origin.
}

For  $D_2$-equivariant maps, additional constraints arise naturally, and we obtain the remaining results.

\begin{lemma}\label{lemaWs}
Let $f\in \Hom(\R^ 2)$ be $C^1$ with symmetry group $D_2$. Suppose that $\Fix(f)=\{0\}$ and $0$ is a local saddle.
If one of the following holds:
\begin{enumerate}
\item[$a)$]
$0$ is a direct saddle;
\item[$b)$]
$\Fix(f^2)\cap \Fix(\Z_2^a)=\{0\}$ and $\Fix(f^2)\cap \Fix(\Z_2^b)=\{0\}$;
\end{enumerate}
then $W^{s}(0,f)=\Fix(\Z_2^j)$ and $W^{u}(0,f)=\Fix(\Z_2^i)$ for  $i\ne j\in\{a,b\}$.
\end{lemma}

\begin{proof}
Since there are  two reflections in  $D_2$, then there exist two $f$-invariant lines,  $\Fix(\Z_2^j)$, $j=a,b$,
containing the origin.
One of the two invariant lines contains the stable global curve and the other the  unstable one.
Without loss of generality, let $W^s(0,f)\subset\Fix(\Z_2^a)=\{(x,0)\ x\in\R\}$.
Let $g(x)$ be the first coordinate of $f(x,0)$.
Then $g\in\Hom(\R)$  is  $\Z_2^b$-equivariant, with  $g(0)=0$ and $\Fix(g)=\{0\}$.
Moreover, $\Fix(g^2)$ is the set of first coordinates of points in $\Fix(f^2)\cap \Fix(\Z_2^a)$.

We prove the result for   $W^s(0,f)$ as the proof for $W^u(0,f)$ is analogous.

In case $a)$ there exists $\alpha>0$ such that in  the interval $[0,\alpha)$, $g$ is a contraction,
hence $0<g(x)<x$.
Because $\Fix(g)=\{0\}$, then $g(x)<x$ for all $x>0$.
Since $g\in\Hom(\R)$, we have $g(x)>0$ for all $x>0$ and the result  for $W^s(0,f)$ follows.

In case $b)$, if the derivative $g^\prime(0)>0$ the proof of case $a)$ holds.
Otherwise, there is $\alpha>0$ such that $g$ maps $[0,\alpha)$ into $(-\alpha,0]$ as a contraction,
hence $-x<g(x)<0$ in that interval.
Then $-x<g(x)$ for all $x>0$ follows from $\Fix(g^2)=\{0\}$, and $g(x)<0$ holds since $g\in\Hom(\R)$,
proving the result for $W^s(0,f)$.
\end{proof}

\begin{theorem} \label{propD2Hirch} Let $f\in \Hom^{+}(\R^ 2)$ be $C^1$ with symmetry group $D_2$. Suppose that $\Fix(f)=\{0\}$ and $0$ is a local  direct saddle. Then 0 is a global saddle. In addition, the global curves $W^s(0,f)$ and $W^u(0,f)$ divide the plane in four connected components that are invariant by $f$.
\end{theorem}

\begin{proof}
Lemma~\ref{propSaddleFree}  ensures that $f$ is free, and it the follows from  Proposition~\ref{freetridyn} that  $f$ has trivial dynamics.
Since there are  two reflections in  $D_2$, then there exist two $f$-invariant lines containing the origin. One of the two invariant lines contains the stable global curve and the other the unstable one.
%
By case $a)$ of Lemma~\ref{lemaWs} the stable global curve
$W^ s$ is the whole of one of the two invariant lines and $W^ u$ is the whole of the other. Hence, there are no homoclinic contacts and $W^{s}(0,f)$, $W^{u}(0,f)$ are unbounded.  Moreover, $\displaystyle{W^s(0,f)\cup W^u(0,f)}$ separates the plane in four connected components that are invariant by $f$.
\end{proof}

{  An example where  Theorem~\ref{propD2Hirch}  can be applied is 
$$
f(x,y)=\left( \frac{2x(1+x^2)}{4+x^2(1+x^2)^2},
\frac{8y(1+x^2)}{4+x^2(1+x^2)^2}
\right) .
$$
Note that $\Fix(f)=\{0\}$ because the first coordinate of $f(x,y)$ equals $x$ if and only if
$$
\left(x^2\left( 1+x^2\right)-2\right)\left( 1+x^2\right)=-4 .
$$
Since the expression $L(x)=x^2\left( 1+x^2\right)-2$ is an increasing function of $x$  satisfying $L(x)\ge -2$, 
and $L(x)>0$ for $x>1 $, then the left hand side cannot ever reach the value  $-4$.
}

\begin{proposition}\label{PD2} Let $f\in \Hom(\R^ 2)$ be of class $C^1$ and with symmetry group $D_2$ such that $0$ is a local saddle. Suppose  for some $\varepsilon>0$,
$\spec(f)\cap [1,1+\varepsilon[=\emptyset$,
and one of the following holds:
\begin{itemize}
\item[$a)$]
$f$ is orientation preserving and $0$ is a local direct saddle;
\item[$b)$]
there exist no $2-$periodic orbits;
\end{itemize}
then $0$ is a global saddle. 
In addition, the global curves $W^s(0,f)$ and $W^u(0,f)$ divide the plane in four connected components that either are $f$-invariant or are interchanged by $f$.
\end{proposition}

\begin{proof}
From Lemma~\ref{corPtoFijo} the origin is the only fixed point of $f$.
If $a)$ holds, then  $0$ is a global saddle
by Theorem \ref{propD2Hirch}.
If $b)$ holds, $0$ is a global saddle for $f^2$ and the proposition follows by Lemma~\ref{lemOmega2} as in Theorem~\ref{teoSaddle}. Notice that, in this case, the global curves $W^s(0,f)$ and $W^u(0,f)$ divide the plane in four connected components that may be interchanged by $f$.
\end{proof}

{    Proposition~\ref{PD2}  can be applied to
$$
f(x,y)=\left(2x(1+y^2),\frac{y}{3}
\right) 
\qquad\mbox{with}\qquad
Df(x,y)=\left(\begin{array}{ll}
2(1+y^2)&4xy\\ 0&1/3
\end{array}\right)
.
$$
The  first coordinate of $f(x,y)$  never equals $x$, so $\Fix(f)=\{0\}$.
Also  $\spec(f)\cap [1,2[=\emptyset$ and $f$ satisfies assumption {\sl a)} of the proposition.
Note that $f$ also satisfies assumption  {\sl b)} of the proposition,
since 
$$
f^2(x,y)=\left(4x(1+y^2)(1+\frac{y^2}{9}),  \frac{y^2}{9}\right) .
$$
}

Note that Example~\ref{exemploTwistedSaddle} is a map with symmetry group $D_2$,
where assumption $(b)$ fails.
So the existence of periodic orbits of period two is also relevant in the presence of $D_2$ symmetry.
This also shows that our hypotheses are minimal.

In the case of $\Z_2(\langle - Id \rangle)$ symmetry we have no reflection and consequently no $f$-invariant line.
We therefore cannot drop the diffeomorphism assumption. This case then proceeds as if there were no symmetries.

\section{Aplications}\label{secApp}

\subsection{Application to the Li\'enard equation} \label{secLienard}

In this section, we illustrate how our results can be used to study differential equations via the Poincar\'e map. The next example was inspired by the Li\'enard equation studied in \cite{CamposTorres} by J. Campos and P. J. Torres.

Consider the differential equation

\begin{equation} \label{lienard}
\ddot x + f(x) \dot x + g(x)=p(t),
\end{equation}
 where $ f, g: \R \to \R$ are locally Lipschitz maps of class $C^1$. Suppose in addition that the following assumptions holds:

 \begin{itemize}
 \item[(A1)] $p:\R \to \R$ is continuous and periodic with minimal period $T>0$;
 \item[(A2)] $f$ is bounded and $f(x)\geq 0$, for all $x\in \R$;
 \item[(A3)] $g$ is a strictly decreasing homeomorphism;
 \item[(A4)] $\exists c,d\geq 0$ such that $\displaystyle{\abs{g(x)}\leq c+d\abs{x}},$ for all $x\in \R$.
 \end{itemize}

The assumptions on $f$, $g$ and $p$ guarantee the existence and uniqueness of solutions  of the initial value problem associated to \eqref{lienard}. 

The solutions of  \eqref{lienard} are the first coordinates of those of:

\begin{equation}\label{vecfieldlienard}
\left\{
\begin{array}{ccc} \dot{x} & = & y-F(x) \\ & & \\
\dot{y} & = &-g(x)+p(t)\end{array}
\right.
\end{equation}
where $\displaystyle{F(x)=\int_{0}^{x} f(s)ds}$.

For each $q\in \R^2$, consider the solution $u(t,q)$ of \eqref{vecfieldlienard} with initial value $u(0)=q$. Let $P(q)=u(T,q)$ be the Poincar\'e map associated to \eqref{vecfieldlienard}. By  uniqueness of solutions, $P$ is well defined and injective. By  continuous dependence on initial conditions, $P$ is continuous.

Since $f$ is bounded, $p$ continuous and periodic and $g$ verifies Assumption (A4), all solutions of \eqref{vecfieldlienard} are defined in the future and in the past. Consequently, $P$ is defined in $\R^2$ and $P(\R^2)=\R^2$. Thus, $P\in \Hom(\R^2)$.

By  differentiable dependence on initial conditions and the Jacobi-Liouville Formula we have
$$
0< \det P'(p)=exp{\int_{0}^{T}\; div_{x}X(t,u(t,p))\; dt},
$$
where $X(t,x,y)=(y-F(x),-g(x)+p(t))$. Hence $P\in \Diff^{+}(\R^2)$.


In addition, by the  sub-supersolution method,  (A1) and (A3) imply the existence of a $T$-periodic solution $u(s)$ of \eqref{lienard}. Actually, by (A1) there exist  $a,b\in \R$ such that $a\leq p(t) \leq b$, $\forall t \in \R$.  Considering  $\beta=g^{-1}(a)$ and $\alpha=g^{-1}(b)$,  Condition (A3) implies that $\forall t \in \R$, $\alpha \leq \beta$ and 
\begin{eqnarray} 
\ddot{\alpha}+f(\alpha)\dot{\alpha}+g(\alpha)\geq p(t) \nonumber\\
\ddot{\beta}+f(\beta)\dot{\beta}+g(\beta)\leq p(t) \nonumber .
\end{eqnarray}

Note that $T$-periodic solutions of \eqref{vecfieldlienard} are fixed points of $P$ and that stability of $T$-periodic solutions corresponds to stability of these fixed points. Hence $P$ has a fixed point.

\begin{proposition}\label{proplienard}
The $T$-periodic solution of \eqref{vecfieldlienard} is unique and 
a global saddle for the Poincar\'e map.
\end{proposition}

\begin{proof}

Let $P$ be the Poincar\'e map associated to \eqref{vecfieldlienard}. Next we show that the periodic solution $u(s)$ of \eqref{vecfieldlienard}  is a direct saddle of $P\in \Diff^{+}(\R^2)$.

Consider the linearisation of \eqref{vecfieldlienard}  around the periodic solution $u(s)$:
$$
\left(\begin{array}{c}\dot x(t)\\ \dot y(t)\end{array}\right)=A(s)
\left(\begin{array}{c}x(t)\\ y(t)\end{array}\right)
\qquad\mbox{where}\qquad
A(s)=\left(\begin{array}{ll}-f(u(s))&1\\ -g^\prime(u(s))&0\end{array}\right).
$$
The eigenvalues of $A(s)$ are $\lambda_{\pm}(s)=\left(-f(u(s))\pm\sqrt{f(u(s))^2-4g^\prime(u(s))}\right)/2$.

Since $g$ is monotonically decreasing and $f(x)\geq 0$ for all $x\in\R$, then
$f(u(s))^2-4g^\prime(u(s))>0$ and $f(u(s))<\sqrt{k^2-4g^\prime(u(s))}$. Consequently,
$\lambda_{+}(s)>0$ and $\lambda_{-}(s)<0$. 
The eigenvalues of the linearisation of $P$ around
the origin are
$\displaystyle{\rm e}^{\int_0^T\lambda_{\pm}(s)ds}$.

Since $u(s)$ corresponds to a local saddle of $P$,  there are solutions of  \eqref{vecfieldlienard} bounded in the future and by \cite[Theorem 3.2]{CamposTorres},  $P$ has a unique fixed point. The proposition follows from Theorem \ref{teoSaddle}.
\end{proof}

\subsection{Equivariant dynamics}

In the symmetric case, we have the following:

\begin{theorem}
Let $\dot x=X(t,x)$ be such that $\gamma X(t,x)=X(t,\gamma x)$ for $\gamma \in O(2)$.
Define the time-$T$ map $P(\xi)=x(T;0,\xi)$, where $x(T;t_0,\xi)$ is the solution satisfying the initial condition $\xi$ at $t=t_0$. 
Then $\gamma P(\xi)=P(\gamma\xi)$.
\end{theorem}

\begin{proof}
It is straightforward to verify that $\gamma x(t;0,\xi)$ and $x(t;0,\gamma\xi)$  satisfy the same initial condition. Hence, they coincide at time $T$.
\end{proof}

Note that if $P$ is to be the Poincar\'e map around a periodic solution and have symmetry $D_2$ then $P$ has a fixed point at the origin. 
Hence, in order to apply our results to a generic differential equation, this has to be first transformed to bring the periodic solution to the origin.
As an illustration of such a transformed system, consider:
$$
\left\{\begin{array}{l}
\dot x= \alpha x+ f_1(x,y)\\
\dot y=-\beta y +f_2(x,y)\\
\dot z=1
\end{array}\right.
\qquad \alpha, \beta >0
$$
such that $f_i(x,y)=O(|(x,y)|^2)$ and $f=(f_1,f_2)$ is $D_2$-equivariant, and either 
$\dot x\ne 0$ or $\dot y\ne 0$ for $(x,y)\ne (0,0)$.
The linear part of $P$ is given by $(x,y)\mapsto\left(e^\alpha x,e^{-\beta}y\right)$
and,
{  by Theorem~\ref{propD2Hirch}, }the origin is a global saddle.

\paragraph{Acknowledgements:}
B. Alarc\'on thanks
Prof. Pedro Torres for his hospitality and fruitful conversations during her stay 
at the University of Granada, Spain. She also thanks Prof. Christian Bonatti for giving her an example which improved her understanding of global saddles. Centro de Matem\'atica da Universidade do Porto (CMUP --- UID/MAT/00144/2013) is funded by FCT (Portugal) with national (MEC) and European structural funds through the programs FEDER, under the partnership agreement PT2020.
B.\ Alarc\'{o}n was also supported in part by CAPES from Brazil and grants MINECO-15-MTM2014-56953-P from Spain and CNPq 474406/2013-0 from Brazil.

%
%
%


\begin{thebibliography}{ZZZZ}


\bibitem{AlarconDenjoy} B. Alarc\'on. Rotation numbers for planar attractors of equivariant homeomorphisms. 
{\em Topological Methods in Nonlinear Analysis}, 42(2), 327--343, 2013.

\bibitem{SofisbeSzlenk} B. Alarc\'on, S.B.S.D Castro and I. Labouriau.  A local but not global attractor for a $\Z_n$-symmetric map.  {\em Journal of Singularities}, 6, 1--14, 2012.


\bibitem{SofisbeGlobal} B. Alarc\'on, S.B.S.D Castro and I. Labouriau.  Global Dynamics for Symmetric Planar Maps.  {\em Discrete \& Continuous Dyn. Syst.} - A, 37, 2241--2251, 2013.


\bibitem{Alarcon} B. Alarc\'on, V. Gu\'{\i}\~nez and C. Gutierrez. Planar Embeddings with a globally
attracting fixed point. {\em Nonlinear Anal.}, 69:(1), 140--150,
2008.

\bibitem{Alarcon-orbitas-periodicas} B. Alarc\'on, C. Gutierrez and J. Mart\'{\i}nez-Alfaro.
Planar maps whose second iterate has a unique fixed point. {\em J.
Difference Equ. Appl.}, 14:(4), 421--428, 2008.


\bibitem{Brown} M. Brown. Homeomorphisms of two-dimensional manifolds. {\em Houston J. Math}, 11(4), 455--469, 1985.


\bibitem{CamposTorres} J.\ Campos and P.J. Torres.
On the structure of the set of bounded solutions on a periodic Li\'enard Equation.
{\em Proceedings of the American Mathematical Society}, 127(5), 1453--1462, 1999

\bibitem{Cima-Manosa} A. Cima, A. Gasull and F. Ma\~nosas. The Discrete
Markus-Yamabe Problem. {\em Nonlinear Analysis}, 35, 343--354, 1999.



\bibitem{vanDenEssen}
A. van den Essen. Polynomial automorphisms and the Jacobian conjecture.
{\em Progress in Mathematics}, 190
Birkh\"auser Verlag, 2000



\bibitem{golu2} M. Golubitsky, I. Stewart and D.G. Schaeffer. Singularities and Groups in Bifurcation Theory Vol. 2. {\em Applied
Mathematical Sciences}, 69, Springer Verlag, 1985.


\bibitem{Grobman} D. M. Grobman, Homeomorphisms of systems of differential equations. {\em Doklady Akad. Nauk SSSR} 128, 880--881, 1959

\bibitem{Gutierrez}
C. Gutierrez, A solution to the bidimensional Global Asymptotic conjecture,  {\em Ann. Inst. Poincar\'e Anal.
Non Lin\'eaire } 12, 627--671, 1995.

\bibitem{sillapepe}
C. Gutierrez, J.  Mart\'inez-Alfaro and  J. Venato-Santos. Plane foliations with a saddle singularity  {\em Topology Appl.} 159(2), 484--491, 2012.

\bibitem{Hartman} P. Hartman. On local homeomorphisms of Euclidean spaces,
{\em Bolet\'\i{}n de la Soc. Matem\'atica Mexicana} 2, 220--241, 1962

\bibitem{sillaHirsch} M. Hirsch. Fixed-Point Indices, Homoclinic Contacts, and Dynamics of Injective Planar Maps. {\em Michigan Math.J.} 47, 101--108, 2000.


\bibitem{Lazer} A. C. Lazer and P. J. McKenna.
On the existence of stable periodic solutions of differential equations of Duffing type.
{\em Proceedings of the American Mathematical Society}, 110(1),  125--133, 1990.

\bibitem{LeCalvez} P. Le Calvez. Une version feuillet\'ee \'equivariante du th\'eor\`eme de translation de {B}rouwer.
{\em Publ. Math. Inst. Hautes \'Etudes Sci.}, 112, 1--98, 2005.

\bibitem{Murthy} P. Murthy. Periodic solutions of two-dimensional forced systems: The Masera Theorem and its extension. {\em J.Dyn and Diff Equations}, 10(2)  275--302, 1998.





\end{thebibliography}
\end{document}